\documentclass{amsart}
\usepackage{amssymb,amsmath,latexsym,amsfonts,amsthm,mathrsfs, lineno}
\usepackage{enumerate}
\newtheorem{thm}{Theorem}
\newtheorem{thmm}{Theorem}
\newtheorem{defn}{Definition}
\newtheorem{prop}{Proposition}

\newtheorem{question}{Question}

\newcommand{\N}{\mathbb{N}}
\newcommand{\Z}{\mathbb{Z}}
\newcommand{\R}{\mathbb{R}}
\newcommand{\Q}{\mathbb{Q}}

\newcommand{\Ur}{\textbf{U}}
\newcommand{\K}{\mathcal{K}}
\newcommand{\funct}[2]{#1 \longrightarrow #2}
\newcommand{\m}[1]{\textbf{#1}}
\newcommand{\mc}[1]{\widetilde{\textbf{#1}}}
\newcommand{\Aut}{\mathrm{Aut}}
\newcommand{\Stab}{\mathrm{Stab}}
\newcommand{\Age}{\mathrm{Age}}
\newcommand{\Emb}{\mathrm{Emb}}

\newcommand{\arrows}[3]{\longrightarrow {#1}^{#2}_{#3}}
\newcommand{\restrict}[2]{#1\mathbin{\upharpoonright} #2}

\author{L. Nguyen Van Th\'e}

\address{Laboratoire d'Analyse, Topologie et Probabilit\'es, Universit\'e d'Aix-Marseille, Centre de Math\'ematiques et Informatique, 
39, rue F. Joliot Curie, 13453 Marseille Cedex 13, France}

\email{lionel@latp.univ-mrs.fr}

\title{Universal flows of closed subgroups of $S_{\infty}$ and relative extreme amenability}


\date{Submitted in March 2012, accepted in July 2012}

\begin{document}

\begin{abstract}

This paper is devoted to the study of universality for a particular continuous action naturally attached to certain pairs of closed subgroups of $S_{\infty}$. It shows that three new concepts, respectively called relative extreme amenability, relative Ramsey property for embeddings and relative Ramsey property for structures, are relevant in order to understand this property correctly. It also allows us to provide a partial answer to a question posed in \cite{KPT} by Kechris, Pestov and Todorcevic.  

\end{abstract}

\maketitle

\section{Introduction}

This note builds on the paper \cite{KPT} by Kechris, Pestov and Todorcevic, and is devoted to the study of universality for a particular continuous action naturally attached to certain pairs of closed subgroups of $S_{\infty}$. Recall that if $G$ is a topological group, a \emph{$G$-flow} is a continuous action of $G$ on a topological space $X$ (in what follows, all topological spaces will be Hausdorff). For those, we will often use the notation $G \curvearrowright X$. The flow $G \curvearrowright X$ is \emph{compact} when the space $X$ is. It is \emph{universal} when every compact minimal $G \curvearrowright Y$ is a factor of $G \curvearrowright X$, which means that there exists $\pi : X \longrightarrow Y$ continuous, onto and $G$-equivariant, i.e.  so that $$\forall g \in G \quad \forall x \in X \quad  \pi (g \cdot x) = g \cdot \pi(x).$$ 

Finally, it is \emph{minimal} when every $x \in X$ has dense orbit in $X$: \[ \forall x \in X \  \  \overline{G\cdot x} = X\] 

It turns out that when $G$ is Hausdorff, there is, up to isomorphism of $G$-flows, a unique $G$-flow that is both minimal and universal. This flow is called \emph{the universal minimal flow} of $G$ and is denoted by  $G \curvearrowright M(G)$. When studying universal minimal flows of closed subgroups of $S_{\infty}$ (here and throughout the paper, $S_{\infty}$ denotes the symmetric group of the natural numbers $\N$, equipped with the pointwise convergence topology), the authors of \cite{KPT} showed that certain flows encode remarkable combinatorial properties, called the Ramsey property and the ordering property. This connection also takes place in a slightly broader context and it is in this more general framework that we will present it here, where pure order expansions (i.e.  order expansions where the language is enriched with a single binary relation symbol, which is always interpreted as a linear ordering) are replaced by precompact relational expansions and where the ordering property is replaced by the expansion property (see \cite{NVT3}). Given closed subgroups $G^{*}\leq  G$ of $S_{\infty}$, there are natural classes $\K$ and $\K^{*}$ of finite objects attached to them, as well as a $G$-flow $G\curvearrowright X^{*}$, where $X^{*}$ is the completion of the quotient $G/G^{*}$ equipped with the projection of the right-invariant metric. When $X^{*}$ is compact, the following properties hold (cf \cite{KPT} for pure order expansions and \cite{NVT3} for the precompact case): 

\begin{enumerate}

\item[i)] The flow $G\curvearrowright X^{*}$ is universal and minimal iff $\K^{*}$ has the Ramsey property and the expansion property relative to $\K$. 

\item[ii)] The flow $G\curvearrowright X^{*}$ is minimal iff $\K^{*}$ has the expansion property relative to $\K$.  

\end{enumerate}

\begin{question}[\cite{KPT}, p.174]

Assume that $\K^{*}$ is a pure order expansion of $\K$. Is universality of $G\curvearrowright X^{*}$ equivalent to the Ramsey property for $\K^{*}$? 

\end{question}

More generally: 

\begin{question}
Assume that $\K^{*}$ is a precompact expansion of $\K$. Is universality of $G\curvearrowright X^{*}$ equivalent to the Ramsey property for $\K^{*}$? 

\end{question}

Apart from the fact that they are very natural in view of the previous results, there is, at least, one other very good reason to study those problems. Indeed, given a class of finite objects, the Ramsey property is very often difficult to prove. Any new tool that would allow to reach it is therefore welcome. 
However, \cite{KPT} already provides a spectacular dynamical characterization of the Ramsey property for $\K ^{*}$, as it shows that it is equivalent to $M(G^{*})$ being reduced to a single point for some subgroup $G^{*}$ of $G$ naturally attached to $\K^{*}$. So what is the real interest of the questions? We started working on them when we realized that the Ramsey property implies universality of $G\curvearrowright X^{*}$. Our hope became then that universality would be equivalent to a strictly weaker combinatorial condition for $\K^{*}$, call it weak Ramsey property for the moment, and that this condition would be, in practice, easier to prove than the Ramsey property. This fact, together with the observation that the expansion property is often much easier to prove than the Ramsey property, would then give an alternate combinatorial approach for further problems involving Ramsey property: in order to prove it for $\K^{*}$, instead of attacking the problem directly, it would suffice to prove that both the weak Ramsey property and the expansion property are satisfied. Indeed, that would prove that the flow $G\curvearrowright X^{*}$ is both universal and minimal, and therefore that $\K^{*}$ has the Ramsey property. So far, this approach is partially successful. It is successful in the sense that universality is indeed equivalent to some combinatorial condition (finally not called weak Ramsey property but, for reasons that should become clear later on, relative Ramsey property) which is, in general, strictly weaker than the Ramsey property. But the success is only partial in the sense that it is unclear whether in practice, this condition is much easier to prove than the Ramsey property.  

Let us now turn to a concise description of the results presented in this paper. Let $L$ be a countable first order signature and $L^{*} = L \cup \{ R_{i} : i\in I^{*}\}$ a countable relational expansion of $L$.  We refer to Section \ref{section:basics} for all other undefined notions. Our first result connects universality of $G\curvearrowright X^{*}$ to a dynamical statement involving actions of $G$ and $G^{*}$.  

\begin{defn}

\label{defn:rea}

Let $H\leq G$ be topological groups. Say that the pair $(G, H)$ is \emph{relatively extremely amenable} when every continuous action of $G$ on every compact space admits an $H$-fixed point. 

\end{defn}

\begin{thm}

\label{thm:univREA}

Let $\m F$ be a Fra\"iss\'e structure in $L$ and $\m F^{*}$ a Fra\"iss\'e precompact relational expansion of $\m F$ in $L^{*}$. The following are equivalent: 

\begin{enumerate}

\item[i)] The flow $G\curvearrowright X^{*}$ is universal.  

\item[ii)] The pair $(G, G^{*})$ is relatively extremely amenable. 

\end{enumerate}

\end{thm}

This result allows us to show that Question 2 has a negative answer by exhibiting concrete examples of classes $\K$ and $\K^{*}$ where the flow $G\curvearrowright X^{*}$ is universal but the Ramsey property does not hold for $\K^{*}$ (see Section \ref{section:univRP}). However, quite surprisingly, we are not able to settle the case of Question 1.

Next, we turn to a combinatorial reformulation of relative extreme amenability. Let $\K$ be a class of $L$-structures and $\K^{*}$ an expansion of $\K$ in $L^{*}$, that is, a class of $L^{*}$-structures such that every element of $\K^{*}$ is an expansion of an element of $\K$. When $\m A, \m B \in \K$, the set of all embedings from $\m A$ into $\m B$ is denoted by  $$\binom{\m B}{\m A}_{Emb}.$$

Let $\m B^{*}$ be an expansion of $\m B$ in $\K^*$ and $a\in \binom{\m B}{\m A}_{Emb}$. The substructure of $\m B^{*}$ supported by $a(\m A)$ is an expansion of $\m A$ in $\K^*$. Using $a$, we can then define an expansion of $\m A$ in $\K^{*}$ as follows: for $i\in I$, call $\alpha(i)$ the arity of $R_{i}$. Then, set $$ \forall i\in I^{*} \quad R^{a}_{i}(x_{1},\ldots,x_{\alpha(i)}) \Leftrightarrow R^{\m B^{*}} _{i}(a(x_{1}),\ldots,a(x_{\alpha(i)})).$$

We will refer to $(a(\m A), \vec R^{a})$ as the \emph{canonical expansion} induced by $a$ on $\m A$. If $a' \in \binom{\m B}{\m A}_{\Emb}$, write $a\cong_{\m B^{*}} a'$ when the canonical expansions on $\m A$ induced by $a$ and $a'$ are equal (not only isomorphic).

\begin{defn}

\label{defn:rRPEmb}

Let $\K$ be a class of finite $L$-structures and $\K^{*}$ an expansion of $\K$ in $L^{*}$. Say that the pair $(\K, \K^*)$ has the \emph{relative Ramsey property for embeddings} when for every $k\in \N$, $\m A \in \K$, $\m B^{*} \in \K^*$, there exists $\m C \in \K$ such that for every coloring $c:\funct{\binom{\m C}{\m A}_{\Emb}}{[k]}$, there exists $b \in \binom{\m C}{\m B}_{\Emb}$ such that: $$ \forall a_{0}, a_{1} \in \binom{\m B}{\m A}_{\Emb} \quad a_{0}\cong_{\m B^{*}} a_{1} \ \Rightarrow \ c(ba_{0}) = c(ba_{1}).$$

\end{defn}


\begin{thm}

\label{thm:univrelativeRP}

Let $\m F$ be a Fra\"iss\'e structure in $L$ and let $\m F^{*}$ be a Fra\"iss\'e precompact relational expansion of $\m F$ in $L^{*}$. Then the following are equivalent: 

\begin{enumerate}

\item[i)] The pair $(G, G^{*})$ is relatively extremely amenable. 

\item[ii)] The elements of $\Age(\m F^{*})$ are rigid and the pair $(\Age(\m F), \Age(\m F^{*}))$ has the relative Ramsey property for embeddings. 
\end{enumerate}

\end{thm}

Then, we investigate the properties of a weakening of the relative Ramsey property for embeddings, which involves only structures, as opposed to embeddings.

\begin{defn}

\label{defn:rRP}

Let $\K$ be a class of finite $L$-structures and $\K^{*}$ an expansion of $\K$ in $L^{*}$. Say that the pair $(\K, \K^*)$ has the \emph{relative Ramsey property for structures} when for every $k\in \N$, $\m A \in \K$ and $\m B^{*} \in \K^*$, there exists $\m C \in \K$ such that for every coloring $c:\funct{\binom{\m C}{\m A}}{[k]}$, there exists $b \in  \binom{\m C}{\m B}_{\Emb}$ such that: $$ \forall \mc A_{0}, \mc A _{1}\in \binom{\m B}{\m A} \ \left(\mc A_{0}\cong_{\m B^{*}} \mc A_{1} \ \Rightarrow \ c(b(\mc A_{0})) = c(b(\mc A_{1}))\right).$$

Above, $\mc A_{0}\cong_{\m B^{*}} \mc A_{1}$ means that $\restrict{\m B^{*}}{\tilde A_{0}} \cong \restrict{\m B^{*}}{\tilde A_{1}}$. 

\end{defn}

This weakening may be strictly weaker than the embedding version, but at the combinatorial level, it is good enough to play the role of a ``weak Ramsey property'' as described above:

\begin{thm}

\label{thm:relRPRP}

Let $\m F$ be a Fra\"iss\'e structure in $L$ and let $\m F^{*}$ be a Fra\"iss\'e precompact expansion of $\m F$ in $L^{*}$ whose age consists of rigid structures. Assume that the pair $(\Age(\m F), \Age(\m F^{*}))$ has the relative Ramsey property for structures and that $\Age(\m F^{*})$ has the expansion property relative to $\Age(\m F)$. Then $\Age(\m F^{*})$ has the Ramsey property. 

\end{thm}

Finally, we show that in order to guarantee the existence of an expansion with both the Ramsey and the expansion property, it is enough to prove the existence of an expansion with the relative Ramsey property for structures:

\begin{thm}

\label{thm:subclass}

Let $\m F$ be a Fra\" iss\'e structure in $L^{*}$ and let $\m F^{*}$ be a Fra\"iss\'e precompact relational expansion of $\m F$ in $L^{*}$. Assume that $\Age(\m F^{*})$ consists of rigid elements and that the pair $(\Age(\m F), \Age(\m F^{*}))$ has the relative Ramsey property for structures. Then $\Age(\m F^{*})$ admits a Fra\"iss\'e subclass with the Ramsey property and the expansion property relative to $\Age(\m F)$.  

\end{thm}

The paper is organized as follows: Section \ref{section:basics} contains all basic notions concerning Fra\"iss\'e theory, structural Ramsey property and precompact relational expansions. Section \ref{section:univ} contains a proof of Theorem \ref{thm:univREA}. Section \ref{section:univRP} exhibits concrete classes answering Question 2. Section \ref{section:relativeRP} provides a proof of Theorem \ref{thm:univrelativeRP}. In Section \ref{section:conseq}, a more detailed study of the relative Ramsey property for embeddings is carried. Finally, Section \ref{section:RPstruct} concentrates on the relative Ramsey property for structures and contains the proofs of Theorem \ref{thm:relRPRP} and Theorem \ref{thm:subclass}. 

\

\textbf{Acknowledgements}: I would like to sincerely thank Miodrag Soki\'c as well as the anonymous referee for their comments which greatly improved the quality of the paper.

\section{Ramsey property, expansion property, precompact expansions}

\label{section:basics}

The purpose of this section is to describe the global framework where our study is taking place. Our main references here are \cite{KPT} for Fra\"iss\'e theory and structural Ramsey theory and \cite{NVT3} for precompact expansions. 

\subsection{Fra\"iss\'e theory}

In what follows, $\N$ denotes the set $\{ 0, 1, 2, \ldots\}$ of natural numbers and for a natural number $m$, $[m]$ denotes the set $\{ 0, \ldots, m-1\}$. We will assume that the reader is familiar with the concepts of first order logic, first order structures, Fra\"iss\'e theory (cf \cite{KPT}, section 2), reducts and expansions (cf \cite{KPT}, section 5). If $L$ is a first order signature and $\m A$ and $\m B$ are $L$-structures, we will write $\m A\leq \m B$ when $\m A$ embeds in $\m B$, $\m A \subset \m B$ when $\m A$ is a substructure of $\m B$ and $\m A \cong \m B$ when $\m A$ and $\m B$ are isomorphic. When $L$ is countable, a \emph{Fra\"iss\'e class} in $L$ will be a countable class of finite $L$-structures of arbitrarily large sizes, satisfying the hereditarity, joint embedding and amalgamation property, and a \emph{Fra\"iss\'e structure} (or \emph{Fra\"iss\'e limit}) in $L$ will be a countable, locally finite, ultrahomogeneous $L$-structure. 

\subsection{Structural Ramsey theory}

In order to define the Ramsey property, let $k, l\in \N$ and $\m A, \m B, \m C$ be $L$-structures. The set of all \emph{copies} of $\m A$ in $\m B$ is written $$ \binom{\m B}{\m A} = \{ \mc A \subset \m B : \mc A \cong \m A\}.$$ 

We use the standard arrow partition symbol $$ \m C \arrows{(\m B)}{\m A}{k, l}$$ to mean that for every map $c: \funct{\binom{\m C}{\m A}}{[k]}$, thought as a $k$-coloring of the copies of $\m A$ in $\m C$, there is $\mc B \in \binom{\m C}{\m B}$ such that $c$ takes at most $l$-many values on $\binom{\mc B}{\m A}$. When $l=1$, this is written  $$ \m C \arrows{(\m B)}{\m A}{k}.$$ 

A class $\mathcal{K}$ of finite $L$-structures is then said to have the \emph{Ramsey property} when $$ \forall k \in \N \quad \forall \m A, \m B \in \mathcal{K} \quad \exists \m C \in \mathcal{K} \quad \m C \arrows{(\m B)}{\m A}{k}.$$

When $\mathcal{K} = \Age (\m F)$, where $\m F$ is a Fra\"iss\'e structure, this is equivalent, via a compactness argument, to: $$ \forall k \in \N \quad \forall \m A, \m B \in \mathcal{K} \quad \m F \arrows{(\m B)}{\m A}{k}.$$

\subsection{Precompact expansions}

Assume now that we have a first-order expansion of $L$, $L^{*} = L\cup \{ R_{i}:i\in I^{*}\}$, with $I^{*}$ countable and every symbol $R_{i}$ relational and not in $L$. Let $\K^{*}$ denote an \emph{expansion} of $\K$ in $L^{*}$ (that means that all elements of $\mathcal{K}^{*}$ are of the form $\m A ^{*} = (\m A, (R_{i} ^{\m A ^{*}})_{i\in I^{*}})$ and that every $\m A \in \mathcal K$ can be enriched to some element $\m A ^{*} = (\m A, (R_{i} ^{\m A ^{*}})_{i\in I^{*}})$ of $\K^{*}$ by adding some relations on $A$). For $\m A^{*} \in \mathcal K ^{*}$, the reduct of $\m A^{*}$ to $L$ is denoted by  $ \restrict{\m A^{*}}{L} = \m A$. Then, $\mathcal K ^{*}$ satisfies the \emph{expansion property} relative to $\mathcal K$ if, for every $ \m A \in \mathcal{K}$, there exists $ \m B \in \mathcal{K}$ such that $$\forall \m A^{*}, \m B^{*} \in\mathcal{K} ^{*} \quad (\restrict{\m A^{*}}{L} = \m A \ \  \wedge \ \ \restrict{\m B^{*}}{L} = \m B) \Rightarrow \m A^{*} \leq \m B^{*}. $$

Next, consider $\m{F}$, a Fra\"{i}ss\'e structure in $L$. For $i\in I^{*}$, the arity of the symbol $R_{i}$ is denoted by  $\alpha(i)$. We let $\m{F}^{*}$ be an expansion of $\m{F}$ in $L^{*}$.  We assume that $\m{F}^{*}$ is also Fra\"{i}ss\'e and write $\m{F}^{*} = (\m{F}, (R^{*}_{i})_{i\in I^{*}})$, or $(\m{F}, \vec{R}^{*})$. We also assume that $\m F$ and $\m F^{*}$ have the set $\N$ of natural numbers as universe. The corresponding automorphism groups are denoted by  $G$ and $G^{*}$ respectively. The group $G^{*}$ will be thought as a subgroup of $G$ and both are closed subgroups of $S_{\infty}$, the permutation group of $\N$ equipped with the topology generated by sets of the form $$ U_{g, F} = \{ h\in G : \restrict{h}{F} = \restrict{g}{F}\},$$ where $g$ runs over $G$ and $F$ runs over all finite subsets of $\N$. Note that the group $S_{\infty}$ admits two natural metrics: a left invariant one, $d_{L}$, defined as $$ d_{L}(g,h) = \frac{1}{2^{m}}, \quad m=\min \{ n\in \N : g(n)\neq h(n)\},$$ and a right-invariant one, $d_{R}$, given by $$d_{R}(g,h) = d_{L}(g^{-1}, h^{-1}).$$

In what follows, we will be interested in the set of all expansions of $\m F$ in $L^{*}$, which we think as the product $$P := \prod_{i\in I^{*}} [2]^{\m{F}^{\alpha(i)}}.$$ 

In this notation, the factor $[2]^{\m{F}^{\alpha(i)}} = \{ 0, 1\}^{\m{F}^{\alpha(i)}}$ is thought as the set of all $\alpha(i)$-ary relations on $\m F$. Each factor $[2]^{\m{F}^{\alpha(i)}}$ is equipped with an ultrametric $d_{i}$, defined by $$ d_{i}(S_{i},T_{i}) = \frac{1}{2^{m}}, \quad m = \min\{ n\in \N : \restrict{S_{i}}{[n]} \neq \restrict{T_{i}}{[n]}\}$$ where $\restrict{S_{i}}{[n]}$ (resp. $\restrict{T_{i}}{[n]}$) stands for $S_{i} \cap [n]^{a(i)}$ (resp. $T_{i} \cap [n]^{a(i)}$). 

The group $G$ acts continuously on each factor as follows: if $i\in I^{*}$, $S_{i}\in [2]^{\m{F}^{\alpha(i)}}$ and $g\in G$, then $g\cdot S_{i}$ is defined by $$\forall y_{1}\ldots y_{\alpha(i)} \in \m F \quad g\cdot S_{i}(y_{1}\ldots y_{\alpha(i)}) \Leftrightarrow S_{i}(g^{-1}(y_{1})\ldots g^{-1}( y_{\alpha(i)})).$$

This allows us to define a continuous action of $G$ on the product $P$ equipped with the supremum distance $d^{P}$ of all the distances $d_{i}$ (where $g\cdot \vec S$ is simply defined as $(g\cdot S_{i})_{i\in I^{*}}$ whenever $\vec S = (S_{i})_{i\in I^{*}} \in P$ and $g\in G$). As a set, $G/G^{*}$ can be thought as $G\cdot\vec{R}^{*}$, the orbit of $\vec{R}^{*}$ in $P$, by identifying $[g]$, the equivalence class of $g$,  with $g\cdot \vec{R}^{*}$. The metric $d_{R}$ induces a metric on the quotient $G/G^{*}$, which coincides with the restriction of $d^{P}$ on $G\cdot\vec{R}^{*}$ (see \cite{NVT3}, Proposition 1). Therefore, we can really think of the metric space $G/G^{*}$ as the metric subspace $G\cdot \vec R^{*}$ of $P$ and it can be shown that the space $G/G^{*}\cong G\cdot \vec R^{*}$ is precompact (i.e. has a compact completion, or, equivalently, a compact closure in $P$) iff every element of $\Age (\m F)$ has finitely many expansions in $\Age(\m F^{*})$ (see \cite{NVT3}, Proposition 2). In that case, we say that\emph{ $\Age(\m F^{*})$ is a precompact expansion of $\Age(\m F)$} (or that $\m F^{*}$ is a precompact expansion of $\m F$).

\section{Relative extreme amenability and universality}

\label{section:univ}

In this section, we prove Theorem \ref{thm:univREA}:

\begin{thmm}


Let $\m F$ be a Fra\"iss\'e structure and $\m F^{*}$ a Fra\"iss\'e precompact relational expansion of $\m F$. The following are equivalent: 

\begin{enumerate}

\item[i)] The flow $G\curvearrowright X^{*}$ is universal for minimal compact $G$-flows. 

\item[ii)] The pair $(G, G^{*})$ is relatively extremely amenable. 

\end{enumerate}

\end{thmm}

\begin{proof}[Proof of $i) \Rightarrow ii)$]

Let $G\curvearrowright X$ be a compact $G$-flow and let $Y\subset X$ be such that $G\curvearrowright Y$ is compact minimal. By universality, find $\pi : \funct{X^{*}}{Y}$ continuous, $G$-equivariant and surjective. Denote $\xi = \pi([e])$, where $e$ denotes the neutral element of $G$ (recall that $X^{*} = \widehat{G/G^{*}}$). Then, $\xi$ is $G^{*}$-fixed, because for $g\in G^{*}$: $$g\cdot\xi = g\cdot\pi([e])=\pi(g\cdot[e]) = \pi([g]) = \pi([e]) = \xi. \qquad \qedhere$$ 

\end{proof}

\begin{proof}[Proof of $ii)\Rightarrow i)$]

Let $G\curvearrowright X$ be a compact minimal $G$-flow. Find $\xi\in X$, $G^{*}$-fixed. Then, let $p : \funct{G}{X}$ be defined by $$p(g) = g\cdot \xi.$$

This map is $G$-equivariant, right uniformly continuous and constant on elements of $G/G^{*}$. Therefore, it induces $\bar p : \funct{G/G^{*}}{X}$, which is also $G$-equivariant and right uniformly continuous. Denote by $\pi$ the continuous extension of $\bar p$ to the completion $X^{*} = \widehat{G/G^{*}}$. Then $\pi$ is $G$-equivariant and surjective because its range is a compact subset of $X$ containing $G\cdot\xi$, which is dense in $X$. \qedhere

\end{proof}

\section{Universality vs Ramsey property}

\label{section:univRP}

A consequence of Theorem \ref{thm:univREA} is that for precompact expansions, universality of $G\curvearrowright X^{*}$ is not equivalent to Ramsey property for $\Age(\m F^{*})$. For example, consider the class $\mathcal U_{S}^{<}$ of finite ordered ultrametric spaces with distances in $S$, where $S$ is a finite subset of $\R$. The corresponding Fra\"iss\'e limit is a countable ordered ultrametric space, denoted by  $(\Ur^{ult}_{S}, <)$. As a linear ordering, it is isomorphic to $(\Q, <)$. Hence, $(\Ur^{ult}_{S}, <)$ can be thought as a precompact relational expansion of $(\Q, <)$, and the group $\Aut(\Ur^{ult}_{S}, <)$ can be thought as a closed subgroup of $\Aut(\Q, <)$. Because this latter group is extremely amenable (see \cite{Pe0}), the pair $(\Aut(\Q, <), \Aut(\Ur^{ult}_{S}, <))$ is relatively extremely amenable and the corresponding flow is universal. However, it is known that $\mathcal U_{S}^{<}$ does not have the Ramsey property, see \cite{NVT1}. A similar situation occurs with finite posets, considering $(\Q, <)$ and the Fra\"iss\'e limit $(\mathbb P, <)$ of the class of all finite ordered posets. This class does not have the Ramsey property (cf \cite{So0}, \cite{So1}), but the corresponding flow is universal. 

In the two previous examples, the pair of groups $(G,H)$ under consideration is proved to be relatively extremely amenable by producing an extremely amenable interpolant, i.e.  an extremely amenable closed subgroup $K$ of $G$ containing $H$. It is a natural question to ask whether every relatively extremely amenable pair of groups admits such an interpolant. The answer in general is negative, as witnessed by the pair $(S_{\infty}, \Aut(\Z, <))$. This result is due to Gutman. 

Finally, in view of the original question posed in \cite{KPT}, we do not know whether universality of $G\curvearrowright X^{*}$ implies Ramsey property of $\Age(\m F^{*})$ when $\m F^{*}$ is a pure order expansion of $\m F$. We believe that the answer should be negative, but were not able to construct any counterexample so far. In fact, results of Soki\'c (see \cite{So0}, \cite{So3}) provide a positive answer in a number of cases. As a possible strategy for a counterexample, start with a Fra\"iss\'e class $\Age(\m F)$ with the Ramsey property and consisting of rigid elements. Consider then the class $\K^{<}$ of all finite order expansions of elements of $\Age(\m F)$ and try to find a Fra\"iss\'e subclass $\mathcal K^{*}\subset \mathcal K^{<}$ without the Ramsey property. Then, calling $\m F^{*}$ the Fra\"iss\'e limit of $\K^{*}$ and denoting $G^{*}=\Aut(\m F^{*})$, we would have $(G, G^{*})$ relatively extremely amenable (because $G$ is extremely amenable), hence $G\curvearrowright X^{*}$ universal, while $\mathcal K^{*}$ does not have the Ramsey property.

\section{Relative extreme amenability and relative Ramsey property for embeddings}

\label{section:relativeRP}

In view of the two previous sections, it is natural to ask whether relative extreme amenability of a pair $(G, G^{*})$ as before can be seen at the level of $\Age(\m F)$ and $\Age(\m F^{*})$. The answer is positive, as shown by the following result. Note that the proof has the same pattern as the proof of Kechris-Pestov-Todorcevic theorem as presented in \cite{NVT3}, Theorem 1.

\begin{thmm}


Let $\m F$ be a Fra\"iss\'e structure and $\m F^{*}$ a Fra\"iss\'e precompact relational expansion of $\m F$. Then the following are equivalent: 

\begin{enumerate}

\item[i)] The pair $(G, G^{*})$ is relatively extremely amenable. 

\item[ii)] The elements of $\Age(\m F^{*})$ are rigid and the pair $(\Age(\m F), \Age(\m F^{*}))$ has the relative Ramsey property for embeddings. 
\end{enumerate}

\end{thmm}

\subsection{Proof of $i)\Rightarrow ii)$}

Assume that $(G, G^{*})$ is relatively extremely amenable. We first prove that all elements of $\Age (\m F^{*})$ are rigid. To do so, consider the set of all linear orderings $LO(\m F)$, seen as a subspace of the space $[2]^{\m F \times \m F}$. The group $G$ acts continuously on this later space via the logic action. The set $LO(\m F)$ is then a $G$-invariant compact subspace. Explicitly, $G$ acts on $LO(\m F)$ as follows: if $ \prec \in LO(\m F)$ and  $g \in G$, then $$ \forall x, y \in \m F \quad x (g\cdot\prec) y \Leftrightarrow g^{-1} (x) \prec g^{-1} (y).$$

By relative extreme amenability of $(G, G^{*})$, there is a $G^{*}$-fixed point in $LO(\m F)$, call it $<^{*}$. Consider now a finite substructure $\m A^{*} \subset \m F^{*}$ and let $\varphi$ be an automorphism of $\m A^{*}$. By ultrahomegeneity of $\m F^{*}$, $\varphi$ extends to an automorphism $\phi$ of $\m F^{*}$. Because $<^{*}$ is $G^{*}$-fixed, it is preserved under $\varphi$. Thus, on $A$, $<^{*}$ is preserved by $\varphi$, which means that $\varphi$ is trivial on $A$. This proves that $\m A^{*}$ is rigid.

To prove that $(\Age(\m F), \Age(\m F^{*}))$ has the relative Ramsey property for embeddings, it suffices to show that $\m F$ satisfies the property that is required for $\m C$ in Definition \ref{defn:rRPEmb}. A compactness argument allows then to find $\m C$. So consider $k\in \N$, $\m A \in \Age(\m F), \m B ^{*}\subset \m F^{*}$ finite and a coloring $$ c : \funct{\binom{\m F}{\m A}_{\Emb}}{[k]}.$$ 

Consider the compact space $[k]^{\binom{\m F}{\m A}_{\Emb}}$, acted on continuously by $G$ by shift: if $\chi \in [k]^{\binom{\m F}{\m A}_{\Emb}}$, $g\in G$ and $a \in \binom{\m F}{\m A}_{\Emb}$, then $$ g\cdot \chi (a) = \chi(g^{-1}a).$$

The set $\overline{G\cdot c}$ is a $G$-invariant compact subspace. By relative extreme amenability of $G$, there is a $G^{*}$-fixed point in $\overline{G\cdot c}$, call it $c_{0}$. The fact that $c_{0}$ is $G^{*}$-fixed means that $c_{0}(a_{0}) = c_{0}(a_{1})$ whenever $a_{0} \cong _{\m F^{*}} a_{1}$. Consider now the finite set $\binom{\m B}{\m A}_{\Emb}$, where $\m B :=\restrict{\m B^{*}}{L}$. Because $c_{0} \in \overline{G\cdot c}$, there is $g \in G$ so that $$\restrict{g\cdot c}{\binom{\m B}{\m A}_{\Emb}} = \restrict{c_{0}}{\binom{\m B}{\m A}_{\Emb}}.$$ 

So if $ a_{0}, a _{1}\in \binom{\m B}{\m A}_{\Emb}$ are such that $ a_{0}\cong_{\m B^{*}} a_{1} $ then $$c(g^{-1}a_{0}) = c(g^{-1}a_{1}).$$

It follows that $g^{-1}$ witnesses the relative Ramsey property for embeddings. $\Box$

\subsection{Proof of $ii)\Rightarrow i)$}

Assume that $\Age (\m F^{*})$ consists of rigid elements and the pair $(\Age(\m F), \Age(\m F^{*}))$ has the relative Ramsey property for embeddings. For $A \subset \N$ finite, we denote by $\Stab(A)$ the pointwise stabilizer $\Stab(A)$ in $G$ and we can make the identification: $$G/\Stab (A) = \binom{\m F}{\m A} _{\Emb}.$$

\begin{prop}

\label{prop:RP0}

Let $k\in \N$, $A\subset \N$ finite and supporting a substructure $\m A$ of $\m F$ and $F\subset G$ finite. Let $\bar f : \funct{G}{[k]}$ be constant on elements of $G/\Stab (A) $. Then there exists $g\in G$ such that$$\forall h, h' \in F \quad h'h^{-1} \in G^{*}\Rightarrow \bar f (gh) = \bar f (gh').$$

\end{prop}

\begin{proof}

The map $\bar f$ induces a map $f : \funct{G/\Stab (A)}{[k]}$, which we may think as a $k$-coloring of $\binom{\m F}{\m A}_{\Emb}$. 
Consider the set $ \{ [h] : h \in F\}$. It is a finite set of embeddings from $\m A$ to $\m F$. Therefore, we can find a finite substructure $\m B\subset \m F$ large enough so that the ranges of all those embeddings are contained in $\m B$. Let $\m B^{*}$ denote the substructure of $\m F^{*}$ supported by $B$. By relative Ramsey property for embeddings applied to $\m A$, $\m B^{*}$ and the coloring $f$, find $g\in G$ such that $$ \forall a_{0}, a _{1}\in \binom{\m B}{\m A}_{\Emb} \ \left(a_{0}\cong_{\m B^{*}} a_{1} \ \Rightarrow \ f(ga_{0}) = f(ga_{1})\right).$$

Consider now $h, h'\in F $ so that $h'h^{-1} \in G^{*}$. Then $[h]\cong_{\m B^{*}}[h']$ and so $$f(g[h]) = f(g[h']).$$ 

At the level of $\bar f$, that means $\bar f (gh) = \bar f (gh')$. $\Box$

\begin{prop}

\label{prop:RP1}

Let $p\in \N$, $f : \funct{G}{\R^{p}}$ left uniformly continuous and bounded (where $\R^{p}$ is equipped with its standard Euclidean structure), $F\subset G$ finite, $\varepsilon > 0$. Then there exists $g \in G$ such that $$ \forall h, h' \in F \quad h'h^{-1} \in G^{*} \Rightarrow \| f(gh) - f(gh')\| < \varepsilon.$$

\end{prop}

\begin{proof}

Let $m\in \N$. Note that as subsets of $G$, elements of $G/\Stab ([m]) $ have diameter $1/2^{m+1}$ with respect to the left invariant metric $d^{L}$ on $G$. Thus, by left uniform continuity, we can find $m\in \N$ large enough so that $f$ is constant up to $\varepsilon$ on each element of $G/\Stab ([m]) $. By local finiteness of $\m F$, let now $A\subset \N$ be finite, supporting a finite substructure $\m A$ of $\m F$ and such that $[m]\subset A$. Then $f$ is also constant up to $\varepsilon$ on each element of $G/\Stab (A)$. Because $f$ is also bounded, we can also find $\bar f : \funct{G}{\R^{p}}$ with finite range, constant on elements of $G/\Stab (A) $ and so that $\| f-\bar f\|_{\infty}<\varepsilon/2$. By Proposition \ref{prop:RP0}, there exists $g\in G$ such that $\bar f (gh) = \bar f(gh') $ whenever $h'h^{-1}\in G^{*}$. Then $\|f(gh) - f(gh')\| < \varepsilon$ whenever $h'h^{-1}\in G^{*}$.      \qedhere 

\end{proof}

We can now show that the pair $(G,G^{*})$ is relatively extremely amenable. Let $G \curvearrowright X$ be a continuous action, with $X$ compact. For $p\in \N$, $\phi : \funct{G}{\R^{p}}$ uniformly continuous and bounded, $F\subset G$ finite, $\varepsilon>0$, set $$ A_{\phi, \varepsilon, F} = \{ x\in X : \forall h \in F\cap G^{*} \quad \| \phi(h\cdot x) - \phi(x)\| \leq \varepsilon\}.$$ 

The family $(A_{\phi, \varepsilon, F})_{\phi, \varepsilon, F}$ is a family of closed subsets of $X$. We claim that it has the finite intersection property. Indeed, if $\phi_{1},\ldots, \phi_{l}, \varepsilon_{1},\ldots, \varepsilon_{l}, F_{1},\ldots, F_{l}$ are given, take $$\phi = (\phi_{1},\ldots,\phi_{l}), \quad \varepsilon = \min (\varepsilon_{1}, \ldots, \varepsilon_{l} ), \quad F = F_{1}^{-1}\cup\ldots\cup F_{l}^{-1}\cup \{e\}.$$ 

Fix $x\in X$ and consider the map $f : \funct{G}{\R^{p_{1}+\ldots+p_{l}}}$ defined by $$ \forall g\in G \quad f(g) = (\phi_{1}(g^{-1}\cdot x),\ldots,\phi_{l}(g^{-1}\cdot x))).$$

Because the maps $\phi_{i}$'s are uniformly continuous and the map $g\mapsto g^{-1}\cdot x$ is left uniformly continuous (cf \cite{Pe}, p40), the map $f$ is left uniformly continuous. By Proposition \ref{prop:RP1}, there exists $g\in G$ so that $$ \forall h, h' \in F \quad h'h^{-1}\in G^{*} \Rightarrow \| f(gh) - f(gh')\|<\varepsilon.$$ 

Equivalently, $$\forall i\leq l \quad \forall h, h'\in F \quad h'h^{-1}\in G^{*} \Rightarrow \|\phi_{i}(h^{-1}g^{-1}\cdot x) - \phi_{i}(h'^{-1}g^{-1}\cdot x )\| \leq \varepsilon_{i} .$$

Taking $x_{0} = g^{-1}\cdot x$ and $h'=e$, we obtain $$\forall i\leq l \quad \forall h\in F_{i} \quad \|\phi_{i}(h\cdot x_{0}) - \phi_{i}(x_{0} )\| \leq \varepsilon_{i} .$$ 

This proves the finite intersection property of the family $(A_{\phi, \varepsilon, F})_{\phi, \varepsilon, F}$. By compactness of $X$, it follows that this family has a non empty intersection. Consider any element $x$ of this intersection. We claim that $x$ is fixed under the action of $G^{*}$: if not, we would find $g\in G^{*}$ so that $g\cdot x\neq x$. Then, there would be a uniformly continuous function $\phi_{0} : \funct{X}{[0,1]}$ so that $\phi_{0}(x)=0$ and $\phi_{0}(g\cdot x)= 1$. That would imply $x\notin A_{\phi_{0}, 1/2, \{g\}}$, a contradiction. \qedhere

\end{proof}

\section{Versions and consequences of the relative Ramsey property}

\label{section:conseq}

\subsection{Canonical expansions}

Recall that when $\m A, \m B \in \K$, $\m B^{*}$ is an expansion of $\m B$ in $\K^*$ and $a\in \binom{\m B}{\m A}_{Emb}$, the canonical expansion induced by $a$ on $\m A$ is the structure $(a(\m A), \vec R^{a})$ defined by $$ \forall i \in I^{*}\quad R^{a}_{i}(x_{1},\ldots,x_{\alpha(i)}) \Leftrightarrow R^{\m B^{*}} _{i}(a(x_{1}),\ldots,a(x_{\alpha(i)})).$$

Note that $a$ is not completely characterized by the canonical expansion it induces on $\m A$ when $a(\m A)$ possesses a non-trivial automorphism, but that it is when $a(\m A)$ is rigid.  In the case where all expansions of $\m A$ are rigid, then the $\cong_{\m B*}$-equivalence classes are those sets of the form $\binom{\m B^{*}}{\m A^{*}}_{\Emb}$, where $\m A^{*}$ ranges over the set of all expansions (possibly isomorphic, but based on $A$) of $\m A$ in $\K^*$.

\subsection{More on the relative Ramsey property for embeddings}

In this section, we present simple facts related to the concept of relative Ramsey property for embeddings. 

\begin{prop}
Assume that the pair $(\mathcal K, \mathcal K^{*})$ has the relative Ramsey property for embeddings, then so does the pair $(\mathcal K, \mathcal K^{**})$ whenever $\mathcal K^{**}$ is an expansion of $\mathcal K$ in $L^{*}$ such that $\mathcal K^{**}\subset \mathcal K^{*}$.   

\end{prop}

\begin{proof}
Direct from the definition. 
\end{proof}

For anyone familiar with Ramsey theory, this is a rather unexpected feature (most of Ramsey type properties are not preserved when passing to subclasses).  


\begin{prop}

Let $\m F$ be a Fra\"iss\'e structure and $\m F^{*}$ a Fra\"iss\'e precompact expansion of $\m F$. Then the pair $(\Age(\m F), \Age(\m F^{*}))$ has the relative Ramsey property for embeddings when for every $k\in \N$, $\m A \in \Age(\m F), \m B^{*} \subset \m F^{*}$ finite and $c:\funct{\binom{\m F}{\m A}_{\Emb}}{[k]}$, there exists $g\in G$ such that: $$ \forall a_{0}, a_{1} \in \binom{\m B}{\m A}_{\Emb} \quad a_{0}\cong_{\m B^{*}} a_{1} \ \Rightarrow \ c(ga_{0}) = c(ga_{1}).$$

\end{prop}

\begin{proof}
A standard compactness argument. 
\end{proof}

Note that because $\m B^{*}$ is a substructure of $\m F^{*}$, $a_{0}\cong_{\m B^{*}} a_{1}$ is equivalent to $a_{0}\cong_{\m F^{*}} a_{1}$, which is equivalent to the existence of $g^{*}\in G^{*}$ so that $a_{1} = g^{*}a_{0}$ (use ultrahomogeneity of $\m F^{*}$). 


\begin{prop}

Let $\m F$ be a Fra\"iss\'e structure and $\m F^{*}$ a Fra\"iss\'e precompact expansion of $\m F$ such that $\Age(\m F^{*})$ consists of rigid elements. Then, the following are equivalent: 

\begin{enumerate}

\item[i)] The pair $(\Age(\m F), \Age(\m F^{*}))$ has the relative Ramsey property for embeddings. 

\item[ii)] For every $k\in \N$, $\m A \in \Age(\m F)$, $\m B^{*} \in \Age(\m F^{*})$, there exists $\m C^{*} \in \Age(\m F^{*})$ such that for every coloring $c:\funct{\binom{\m C}{\m A}_{\Emb}}{[k]}$, there exists $b \in \binom{\m C}{\m B}_{\Emb}$ such that: $$ \forall a_{0}, a_{1} \in \binom{\m B}{\m A}_{\Emb} \quad a_{0}\cong_{\m B^{*}} a_{1} \ \Rightarrow \ (ba_{0}\cong_{\m C^{*}} ba_{1} \ \wedge \ c(ba_{0}) = c(ba_{1})).$$

\end{enumerate}

\end{prop}

\begin{proof}

Quite clearly, item $ii)$ implies the relative Ramsey property for embeddings. For the converse, use the standard trick of enriching the coloring $c$ by the $\cong_{\m C^{*}}$-isomorphism type. Formally, fix $k\in \N$, $\m A \in \Age(\m F)$, $\m B^{*} \in \Age(\m F^{*})$. Recall that $E(\m A)$ denotes the set of all (possibly isomorphic) expansions of $\m A$ in $\Age(\m F^{*})$. Consider $\m C$ provided by the relative Ramsey property for embeddings applied to $\m A$, $\m B^{*}$ and $k|E(\m A)|$. Let $\m C^{*}$ denote any expansion of $\m C$ in $\Age(\m F^{*})$. Let $c:\funct{\binom{\m C}{\m A}_{\Emb}}{[k]}$ and define $\bar  c$ by $\bar c(a) = (c(a), [a])$, where $[a]$ denotes the $\cong_{\m C^{*}}$ equivalence class of $a$. This is a $k|E(\m A)|$-coloring of $\binom{\m C}{\m A}_{\Emb}$, so we can find $b\in \binom{\m C}{\m B}_{\Emb}$ so that $$ \forall a_{0}, a_{1} \in \binom{\m B}{\m A}_{\Emb} \quad a_{0}\cong_{\m B^{*}} a_{1} \ \Rightarrow \ \bar c(ba_{0}) = \bar c(ba_{1}).$$

We are done since $$ \bar c(ba_{0}) = \bar c(ba_{1})   \Leftrightarrow \ (ba_{0}\cong_{\m C^{*}} ba_{1} \ \wedge \ c(ba_{0}) = c(ba_{1})). \qquad \qedhere$$

\end{proof}


We now turn to a consequence of the relative Ramsey property for embeddings. Let $\m A \in \Age(\m F)$ and $\m B^{*}\in \Age(\m F^{*})$. Recall that the set $\binom{\m B}{\m A}_{\Emb}$ is partitioned into $\cong_{\m B*}$-equivalence classes, which correspond to those sets of the form $\binom{\m B^{*}}{\m A^{*}}_{\Emb}$, where $\m A^{*}$ ranges over the set of all (possibly isomorphic) expansions of $\m A$ in $\Age(\m F^{*})$ that are based on $A$. As a direct consequence:

\begin{prop}

Let $\m F$ be a Fra\"iss\'e structure and $\m F^{*}$ a Fra\"iss\'e precompact expansion of $\m F$ such that $\Age(\m F^{*})$ consists of rigid elements. Assume that the pair $(\Age(\m F), \Age(\m F^{*}))$ has the relative Ramsey property for embeddings. Then, for every $k\in \N$, $\m A^{*}, \m B^{*} \in \Age(\m F^{*})$, there exists $\m C \in \Age(\m F)$ such that for every coloring $c:\funct{\binom{\m C}{\m A}_{\Emb}}{[k]}$, there exists $b \in \binom{\m C}{\m B}_{\Emb}$ such that the set $b \circ \binom{\m B^{*}}{\m A^{*}}_{\Emb} := \{ ba^{*} : a^{*} \in \binom{\m B^{*}}{\m A^{*}}_{\Emb} \}$ is monochromatic. 




\end{prop}

Quite surprisingly, the converse to the previous proposition does not seem to hold. The main obstruction is that we only have a limited control on how $b$ behaves with respect to canonical expansions. As we have seen before, we can make sure that $b$ preserves $\cong _{\m B^{*}}$-equivalence. However, we cannot make sure that it preserves canonical expansions. This detail appears to be problematic when trying to deduce the relative Ramsey property for embeddings from a repeated application of item $ii)$.

\section{The relative Ramsey property for structures}

\label{section:RPstruct}

Let $\K$ be a class of $L$-structures and $\K^{*}$ an expansion of $\K$ in $L^{*}$. Recall that the pair $(\K, \K^*)$ has the \emph{relative Ramsey property for structures} when for every $k\in \N$, $\m A \in \K$ and $\m B^{*} \in \K^*$, there exists $\m C \in \K$ such that for every coloring $c:\funct{\binom{\m C}{\m A}}{[k]}$, there exists $b \in  \binom{\m C}{\m B}_{\Emb}$ such that: $$ \forall \mc A_{0}, \mc A _{1}\in \binom{\m B}{\m A} \ \left(\mc A_{0}\cong_{\m B^{*}} \mc A_{1} \ \Rightarrow \ c(b(\mc A_{0})) = c(b(\mc A_{1}))\right),$$ where $\mc A_{0}\cong_{\m B^{*}} \mc A_{1}$ means that $\restrict{\m B^{*}}{\tilde A_{0}} \cong \restrict{\m B^{*}}{\tilde A_{1}}$.

Again, when $\K$ and $\K^{*}$ are of the form $\Age(\m F)$ and $\Age(\m F^{*})$ respectively, where $\m F$ and $\m F^{*}$ are Fra\"iss\'e structures and $\m F^{*}$ is an expansion of $\m F$, a compactness argument shows that $(\Age(\m F), \Age(\m F^{*}))$ has the relative Ramsey property for structures when for every $k\in \N$, $\m A \in \Age(\m F)$, $\m B^{*} \in \Age(\m F^{*})$ and $c:\funct{\binom{\m F}{\m A}}{[k]}$, there exists $g\in G$ such that: $$ \forall \mc A_{0}, \mc A _{1}\in \binom{\m B}{\m A} \ \left(\mc A_{0}\cong_{\m B^{*}} \mc A_{1} \ \Rightarrow \ c(g(\mc A_{0})) = c(g(\mc A_{1}))\right).$$

This property implies (but does not seem to be equivalent to) the following weakening of the Ramsey property for $\Age(\m F^{*})$: for every $k\in \N$, $\m A^{*}, \m B^{*} \in \Age(\m F^{*})$ and $c:\funct{\binom{\m F}{\m A}}{[k]}$, there exists $g\in G$ such that $g\circ\binom{\m B^{*}}{\m A^{*}}$ is monochromatic.

Note that colorings of structures can be seen as particular cases of colorings of embeddings, where elements with isomorphic (and not necessarily equal) canonical expansions receive the same color. For that reason, the relative Ramsey property for embeddings implies the relative Ramsey property for structures.  The converse does not seem to hold in general. The only instance for which we could check that the two notions agree is when the elements of $\Age(\m F)$ are rigid, simply because the sets $\binom{\m F}{\m A}$ and $\binom{\m F}{\m A}_{\Emb}$ can be identified. However, from the practical point of view, the following results show that the relative Ramsey property for structures may have some applications in the future.

\begin{prop}

\label{prop:relRPsub}

Assume that the pair $(\mathcal K, \mathcal K^{*})$ has the relative Ramsey property for structures, then so does the pair $(\mathcal K, \mathcal K^{**})$ whenever $\mathcal K^{**}\subset \mathcal K^{*}$.   
\end{prop}

\begin{proof}
Direct from the definition. 
\end{proof}

\begin{thmm}

Let $\m F$ be a Fra\"iss\'e structure and $\m F^{*}$ a precompact expansion of $\m F$ whose age consists of rigid structures. Assume that the pair $(\Age(\m F), \Age(\m F^{*}))$ has the relative Ramsey property for structures and that $\Age(\m F^{*})$ has the expansion property relative to $\Age(\m F)$. Then $\Age(\m F^{*})$ has the Ramsey property. 

\end{thmm}

\begin{proof}

Because of the relative Ramsey property for structures, every $\m A \in \Age(\m F)$ has a finite Ramsey degree in $\Age(\m F)$ whose value is at most equal to the number of non-isomorphic expansions of $\m A$ in $\Age(\m F^{*})$. Together with the expansion property relative to $\Age(\m F)$, this is known to imply the Ramsey property for $\Age(\m F^{*})$ (for a reference, see for example \cite{NVT3}, Section 5, Proposition 8).  \qedhere

\end{proof}

The preceding result could turn out to be useful in practice, where the Ramsey property is often difficult to prove and the expansion property generally more accessible.

%
%
%
%
%
%

\begin{thmm}

Let $\m F$ be a Fra\" iss\'e structure and let $\m F^{*}$ be a precompact relational expansion of $\m F$. Assume that $\Age(\m F^{*})$ consists of rigid elements and that the pair $(\Age(\m F), \Age(\m F^{*}))$ has the relative Ramsey property for structures. Then $\Age(\m F^{*})$ admits a Fra\"iss\'e subclass with the Ramsey property and the expansion property relative to $\Age(\m F)$.  

\end{thmm}

\begin{proof}

Let $\vec S \in \overline{G\cdot \vec R^{*}}$ be such that $G\curvearrowright\overline{G\cdot \vec S}$ is minimal. Note that $\Age(\m F, \vec S)\subset \Age(\m F^{*})$ because $\vec S \in  \overline{G\cdot \vec R^{*}}$ (cf \cite{NVT3}, section on minimality). We claim that $\Age(\m F, \vec S)$ is a required. This class clearly has the hereditarity property and the joint embedding property. The expansion property comes from minimality of $G\curvearrowright\overline{G\cdot \vec S}$ (cf \cite{NVT3}, remark following Theorem 4 in Section 4). To prove the Ramsey property, notice first that because $\Age(\m F, \vec S)\subset \Age(\m F^{*})$, the pair $(\Age(\m F), \Age(\m F, \vec S))$ also has the relative Ramsey property for structures (cf Proposition \ref{prop:relRPsub}), which in turn implies that every $\m A \in \Age(\m F)$ has a finite Ramsey degree at most equal to the number of non-isomorphic expansions of $\m A$ in $\Age(\m F, \vec S)$. Because $\Age(\m F, \vec S)$ has the hereditary property, the joint embedding property and the expansion property relative to $\Age(\m F)$, it has the Ramsey property (\cite{NVT3}, Section 5, Proposition 8). Finally, because of all the previous properties and because $\Age(\m F, \vec S)$ consists of rigid elements, it is Fra\"iss\'e (\cite{KPT} p.20, or \cite{NVT3} end of Section 6). \qedhere 

\end{proof}

Further investigation about the practical status of the relative Ramsey property for structures will decide on its value as a tool to derive the Ramsey property. We have to admit that, so far, we are not aware of any concrete application of any of the preceding results of this section.

\bibliographystyle{amsalpha}
\bibliography{bib}

\end{document}